\documentclass[oneside]{amsart}
\usepackage[textwidth=13.6cm,textheight=22cm]{geometry} %to go with margin comments --- change textwidth to 15cm?

\usepackage{amsmath}
\usepackage{amssymb}
\usepackage{amsthm}
\usepackage{amscd}
\usepackage{enumerate}

\newcommand{\m}{^{-1}}
\usepackage{hyperref}
\usepackage[margin=1.5cm,small]{caption} % to reduce width and fontsize of captions
\usepackage{color}
\newcommand{\varespilon}{\varepsilon}

\theoremstyle{plain}

\newtheorem{thmspecial}{Theorem}

\newtheorem{thm}{Theorem}[section]
\newtheorem{prop}[thm]{Proposition}
\newtheorem{lem}[thm]{Lemma}

\theoremstyle{definition}
\newtheorem{defn}{Definition}
\theoremstyle{remark}

\DeclareMathOperator{\Hamm}{Hamm}

\DeclareMathOperator{\Sym}{Sym}

\DeclareMathOperator{\id}{Id}

\DeclareMathOperator{\Aut}{Aut}

\DeclareMathOperator{\supp}{Supp}\newcommand{\Supp}{\supp}

\numberwithin{thm}{section}

\newcommand{\modulus}[1]{\left|#1\right|}

\begin{document}
\title{The Wreath Product of Two Sofic Groups is Sofic}      % Enter your title between curly braces
\author{Ben Hayes and Andrew Sale}
\address{ Vanderbilt University\\
         Nashville, TN 37240}
\email{benjamin.r.hayes@vanderbilt.edu, andrew.sale@some.oxon.org}
\date{\today}

\begin{abstract} 
	Given sofic approximations for countable, discrete groups $G,H$, we construct a sofic approximation for their wreath product $G\wr H$. \end{abstract}

\maketitle
%\tableofcontents

Sofic groups, introduced by Gromov \cite{GromovSofic} and developed by Weiss \cite{WeissSofic}, are a large class of groups which can be approximated, in some sense, by finite groups.

There are many examples, including all amenable groups, all residually finite groups, and all linear groups (by Malcev's Theorem).
However, because of the weakness of the approximation by finite groups, few permanence properties of soficity are properly understood. 
Relatively straight-forward examples include closure under direct product and increasing unions, and the soficity of residually sofic groups.
More substantial results generally require some amenability assumption.
For example, an amalgamated product of two sofic groups is know to be sofic if the amalgamated subgroup is amenable (see \cite{ESZ2},\cite{LPaun},\cite{DKP},\cite{PoppArg}). 
This was extended to encompass the fundamental groups of all graphs of groups with sofic vertex groups and amenable edge groups \cite{CHR}.
In the same paper, it is shown that the graph product of sofic groups is sofic.
Also, if $H$ is sofic and is a coamenable subgroup of $G$, then $G$ is sofic too \cite{ESZ1}.

We prove a new permanence result, namely that soficity is closed under taking wreath products:

\begin{thmspecial}\label{thmspecial:sofic wreath}
	Let $G,H$ be countable, discrete, sofic groups. Then $G\wr H$ is sofic.
\end{thmspecial}

We remark that our result is general and requires no amenability or residual finiteness assumptions. 
The special case of Theorem \ref{thmspecial:sofic wreath} when $G$ is abelian was proved by Paunescu \cite{LPaun}, who used methods of analysis and the notion of sofic equivalence relations developed by Elek and Lippner \cite{ElekLip}.
While finishing this paper, we learnt that Holt and Rees have dealt with the case when $H$ is residually finite and $G$ sofic \cite{HR}.
We prove the general result directly, by constructing a sofic approximation for the wreath product, giving a proof that is constructive, quantitative (see Proposition \ref{prop:sofic approx}), and entirely self-contained.

The notion of hyperlinearlity gives a class of groups defined in a similar vein to sofic groups, but where they are approximated instead by unitary groups. Our construction extends to the situation where $G$ is hyperlinear and $H$ sofic, showing that $G\wr H$ is hyperlinear, see \cite{BHhomepage}.

	  Via their approximations by finite groups, sofic groups have applications to problems of current mathematical interest in a wide area of fields. 
	  Sofic groups are relevant to ergodic theory because they are the largest class of groups for which Bernoulli shifts are classified by their base entropy (see \cite{Bow},\cite{KLi}) and for which Gottschalk's surjuncitivity conjecture holds (see e.g. \cite{GromovSofic},\cite{KLi}).
	  In the study of group rings they are useful because they are the largest class of groups for which Kaplansky's direct finiteness conjecture (see \cite{ESZKaplansky}) is known.
	  In the field of $L^{2}$--invariants, they are the largest class of groups for which the determinant conjecture is known (see \cite{ElekSzaboDeterminant}), which is necessary to define $L^{2}$--torsion (see \cite{Luck} Conjecture 3.94). They are also the largest class of groups for which an analogue of L\"{u}ck approximation is known (see \cite{Thom}).
	  We refer the reader to \cite{Luck} for applications of $L^{2}$--invariants to geometry and group theory. See also \cite{Pestov},\cite{CapLup} for a survey of sofic groups.

\textbf{Acknowledgments.} The first named author would like to thank Jesse Peterson for asking him if wreath products of sofic groups are sofic at the NCGOA Spring Institute in 2012 at Vanderbilt University.

\section{Preliminaries}\label{sec:prelims}

We begin with the necessary definitions, as well as a useful lemma to help us identify sofic approximations in wreath products.

\begin{defn} Let $A$ be a finite set. The \emph{normalized Hamming distance}, denoted $d_{\Hamm},$ on $\Sym(A)$ is defined by
\[d_{\Hamm}(\sigma,\tau)=\frac{1}{|A|}|\{a\in A:\sigma(a)\ne \tau(a)\}|.\]
\end{defn}

\begin{defn} Let $G$ be a countable discrete group, $F$  a finite subset of $G$, and $\varespilon>0.$ Fix a finite set $A$ and a function $\sigma\colon G\to \Sym(A).$ We say that $\sigma$ is \mbox{\emph{$(F,\varepsilon)$--multiplicative}} if
\[\max_{g,h\in F}d_{\Hamm}(\sigma(g)\sigma(h),\sigma(gh))<\varepsilon.\]
 We say that $\sigma$ is \emph{$(F,\varepsilon)$--free} if
\[\min_{g\in F\setminus\{1\}}d_{\Hamm}(\sigma(g),\id)> 1-\varepsilon.\]
We say that $\sigma$ is an \emph{$(F,\varepsilon)$--sofic approximation} if it is $(F,\varepsilon)$--multiplicative,  $(F,\varepsilon)$--free,  and furthermore $\sigma(1)=\id.$ 
Lastly, we say that $G$ is \emph{sofic} if for every finite $F\subseteq G$ and $\varepsilon>0,$ there is a finite set $A$ and an $(F,\varepsilon)$--sofic approximation $\sigma\colon G\to \Sym(A).$
\end{defn}

Our aim is to use sofic approximations for $G$ and $H$ and build a sofic approximation for $G\wr H$.
First recall that the wreath product is defined as 
		$$G\wr H = \bigoplus\limits_{H}G \rtimes H$$
where the action of $h\in H$ is given via $\alpha_{h}\in \Aut\left(\bigoplus_{H}G\right)$, defined by
		\[\alpha_{h}\Big( (g_{x})_{x\in H} \Big)=(g_{h\m x})_{x\in H}.\]
A homomorphism $\pi\colon G\wr H\to K$, for some group $K$, can be decomposed into a pair of homomorphisms $\pi_{1}\colon \bigoplus_{H}G\to K$, $\pi_{2}\colon H\to K$ which satisfy the following equivariance condition:
\[\pi_{2}(h)\pi_{1}(g)=\pi_{1}(\alpha_{h}(g))\pi_{2}(h),\mbox{ for all $h\in H,g\in \bigoplus_{H}G.$}\]
The following lemma gives an ``approximate analogue'' to this situation.

\begin{lem}\label{l:almosthomom} Let $G,H$ be countable, discrete groups. 
For every finite set $F_0\subseteq G\wr H$  there are finite  sets $E_{1}\subseteq \bigoplus_{H}G,E_{2}\subseteq H$ 
such that the following holds:
Let $\varepsilon>0$ and $\Omega$ be a finite set. Suppose $\sigma\colon G\wr H\to \Sym(\Omega)$ is a map such that  %\Bmodif 

\begin{itemize}
	\item the restriction of $\sigma$ to $\displaystyle \bigoplus_{H}G$ is $(E_{1},\varepsilon/6)$--multiplicative,
	\item the restriction of $\sigma$ to $H$ is $(E_{2},\varepsilon/6)$--multiplicative, \phantom{$\displaystyle \bigoplus_{H}G$} %added for visual effect
	\item $\displaystyle \max_{g\in E_{1},h\in E_{2}}d_{\Hamm}\big(\sigma(g,h),\sigma(g,1)\sigma(1,h)\big)<\varepsilon/6$,
	\phantom{$\displaystyle \bigoplus_{H}G$} %added for visual effect
	\item $\displaystyle \max_{g\in E_{1},h\in E_{2}}d_{\Hamm}\left(\sigma(1,h)\sigma(g,1),\sigma(\alpha_{h}(g),1)\sigma(1,h)\right)<\varepsilon/6$.
	\phantom{$\displaystyle \bigoplus_{H}G$} %added for visual effect
\end{itemize}
Then $\sigma$ is $(F_0,\varepsilon)$--multiplicative.

\end{lem}

\begin{proof}

After making the right definitions for $E_1,E_2$, we apply the triangle inequality several times to obtain the result.
We require that if $(g,h),(\hat{g},\hat{h})$ are in $F_0$, then  $g, \hat{g},\alpha_h(\hat{g}) \in E_1$, and $h,\hat{h}\in E_2$. 
This is true if we define 
		\begin{align*}
		E_1 &= \big\{ \alpha_h (g) : h\in \pi_H(F_0) \cup \{1\} , g\in \pi_G(F_0) \big\}, \\
		E_2 &= \pi_H(F_0).
		\end{align*}
 We leave verification that this is sufficient to the reader.
%$E_1 = \pi_G((F_0 \cup \{1\})^{2})$ and $E_2 = \pi_H((F_0\cup \{1\})^{2})$.
\end{proof}

To see how this gives an ``approximate analogue'' of the situation for homomorphisms, notice that the above lemma  says that an approximate homomorphism $\sigma\colon G\wr H\to \Sym(A)$ can be thought of as a pair of approximate homomorphisms $\sigma_{1}\colon \bigoplus_{H}G\to \Sym(A)$, $\sigma_{2}\colon H\to\Sym(A)$ so that
\[\sigma_{2}(h)\sigma_{1}(g) \approx \sigma_{1}(\alpha_{h}(g))\sigma_{2}(g)\]
for all $g$ in a large enough finite subset of $\bigoplus_{H}G$ and all $h$ in a large enough finite subset of $H.$ (Here we use $\approx$ to indicate that both maps agree on a sufficiently large subset of $A$). It is not hard to see that the above lemma is valid with $G\wr H$ replaced with any semidirect product and $\Sym(A)$ replaced with any group equipped with a bi-invariant metric.

\section{The sofic approximation for wreath products}\label{sec:wreath sofic}

To facilitate our proof, we need to introduce some notation. Let $G,H$ be countable discrete groups and $\sigma_{A}\colon G\to \Sym(A),\ \sigma_{B}\colon H\to \Sym(B)$ be two functions (not assumed to be homomorphisms). For $h\in H,b_{0}\in B,$ define
\[\sigma_{A,b_{0}}^{(h)}\colon G\to \Sym\left(\bigoplus_{B}A\right)\]
by
\[\sigma_{A,b_{0}}^{(h)}(g)\Big((a_{b})_{b\in B}\Big)=(\widehat{a}_{b})_{b\in B},\]
where
\[\widehat{a}_{b}=\begin{cases}
\sigma_{A}(g)(a_{b}),& \textnormal{ if $b=\sigma_{B}(h)b_{0}$,}\\
a_{b},& \textnormal{ otherwise.}
\end{cases}\]
Now suppose that $E\subseteq H$ is finite, and take  a subset 
\[B_{0}\subseteq \{b\in B:\sigma_{B}(h_{1})b\ne \sigma_{B}(h_{2})b\mbox{ for all $h_{1},h_{2}\in E$ with $h_{1}\ne h_{2}.$}\}.\]
Note that $\sigma_{A,b_{0}}^{(h_{1})}(g_{1})$ and $\sigma_{A,b_{0}}^{(h_{2})}(g_{2})$ commute if $b\in B_{0},g_{1},g_{2}\in G,h_{1},h_{2}\in H$ and $h_{1}\ne h_{2}.$ Thus it makes sense to define, for $b\in B_{0},$
\[\sigma_{A,b}\colon \bigoplus_{E}G\to \Sym\left(\bigoplus_{B}A\right)\]
by
\[\sigma_{A,b}\Big((g_{x})_{x\in E}\Big)=\prod_{x\in E}\sigma_{A,b}^{(x)}(g_{x}).\]
In our applications $\sigma_{B}$ will be a sofic approximation, so we can take $B_{0}$ to  make up the majority of $B$. Thus $\sigma_{A,b}$ will be defined for ``most" $b\in B.$ We package all these maps together as a single map
\[\widehat{\sigma_{A}}\colon \bigoplus_{E}G\to \Sym\left(\left(\bigoplus_{B}A\right)\oplus B\right)\]
by
\[\widehat{\sigma_{A}}(g)(a,b)=\begin{cases}
\big(\sigma_{A,b}(g)(a),b\big),& \textnormal{if $b\in B_{0}$}\\
(a,b),& \textnormal{if $b\in B\setminus B_{0}$.}
\end{cases}\]
We extend $\widehat{\sigma_{A}}$ to $\bigoplus_{H}G$ by declaring that $\widehat{\sigma_{A}}(g)=\id$ if $g\in \bigoplus_{H}G,$ but $g\notin \bigoplus_{E}G.$ Now define
\[\widehat{\sigma_{B}}\colon H\to \Sym\left(\left(\bigoplus_{B}A\right)\oplus B\right)\]
by
\[\widehat{\sigma_{B}}(h)(a,b)=\big(a,\sigma_{B}(h)b\big).\]
 Finally, define
\[\widehat{\sigma}\colon G\wr H\to \Sym\left(\left(\bigoplus_{B} A\right)\oplus B\right)\]
by
\[\widehat{\sigma}(g,h)=\widehat{\sigma_{A}}(g)\widehat{\sigma_{B}}(h).\]
 Note that $\widehat{\sigma}$ is determined by the choice of $E,B_{0}$ and the two maps $\sigma_A$ and $\sigma_B$.

\begin{prop}\label{prop:sofic approx}
	Let $F\subseteq G\wr H$ be finite and $\varepsilon>0.$ Then there are finite sets $E_A\subseteq G,E_B\subseteq H$ and an $\varepsilon'>0$ so that if $\sigma_A\colon G\to \Sym(A)$ and $\sigma_B\colon H\to \Sym(B)$ are $(E_A,\varepsilon'),(E_B,\varepsilon')$--sofic approximations for $G,H$ respectively, then $\widehat{\sigma}$ is an $(F,\varepsilon)$--sofic approximation.
\end{prop}

The remainder of this section is dedicating to proving Proposition \ref{prop:sofic approx}.
We will see below that it is possible to compute an explicit upper bound on $\varepsilon'$.
It will depend only on $\varepsilon$ and the set $F$.

We remark that $\widehat{\sigma}(1,1)=\id$ by construction. We need to show it is $(F,\varepsilon)$--multiplicative and $(F,\varepsilon)$--free.
First we explain how to define the sets $E$, $E_A$ and $E_B$.

Let $F\subseteq G\wr H$ be finite and $\varepsilon>0$.
Define projections $\pi_{G}\colon G\wr H\to \bigoplus_{H}G$ and $\pi_{H}\colon G\wr H\to H$ by $\pi_{G}(g,h)=g,\pi_{H}(g,h)=h$.
Let $E_{1},E_{2}$ be as in Lemma \ref{l:almosthomom} for the finite set $F_0=F\cup \{1\} \cup F\m$.
As in the proof of Lemma \ref{l:almosthomom}, we have
		\begin{align*}
		E_1 &= \big\{ \alpha_h (g) : h\in \pi_H(F_0) , g\in \pi_G(F_0) \big\}, \\
		E_2 &= \pi_H(F_0).
		\end{align*}
Recall that for $g=(g_{x})_{x\in H}\in \oplus_{H}G$ the support of $g,$ denoted $\supp(g),$ is the set of $x\in H$ with $g_{x}\ne 1.$  We set
\[E = E_2 \cup \bigcup_{\substack{g\in E_1\\h\in E_2}}h\supp(g),\]
\[E_A = \big\{ g_x \in G : (g_x)_{x\in H} \in E_1\big\},\] 
\[E_{B}=E\m E .\]
Then our collections of finite sets satisfy the following properties, each of which we need later on:
\begin{eqnarray*}
\label{condition E_A} E_A &\supseteq & \{ g_x : (g_x)\in E_1, x\in E\},\\
\label{condition E} E\phantom{_1}  & \supseteq & h\Supp(g) \mbox{ for all $h\in E_2, g \in E_1$},\\
\label{condition E_2} E\phantom{_1}  &  \supseteq & E_2,\\
\label{condition E_B} E_{B} & \supseteq & E\cup E^{-1}\cup E\m E.	
\end{eqnarray*} 
Choose $\varepsilon'$ so that  $0<\varepsilon'<\frac{\varepsilon}{48\modulus{E}^2}$. 
Let $\sigma_{A}\colon G\to \Sym(A),\ \sigma_{B}\colon H\to \Sym(B)$ be $(E_A,\varepsilon'),(E_B,\varepsilon')$--sofic approximations respectively.
Set
\[B_{01}=\{b\in B:\sigma_B(h_{1})b\ne \sigma_B(h_{2})b\mbox{ for all $h_{1},h_{2}\in E,h_{1}\ne h_{2}$}\},\]
\[B_{02}=\{b\in B:\sigma_B(h_{1}h_{2})b=\sigma_B(h_{1})\sigma_B(h_{2})b\mbox{ for all $h_{1},h_{2}\in E$}\},\]
\[B_{0}=B_{01}\cap B_{02}.\]
Since $\sigma_B$ is a sofic approximation, we can intuitively think of $B_0$ as making up most of the set $B$. 
Indeed, Lemma \ref{lem:prelim obs} below confirms this.
Use the sets $E \subset H,B_{0}\subseteq B$ and the maps $\sigma_A,\sigma_B$ to define the maps $\widehat{\sigma_A},\widehat{\sigma_B},\widehat{\sigma}$, as constructed at the start of this section. We claim that the map $\widehat{\sigma}$ is an $(F,\varepsilon)$--sofic approximation.
We first make the following preliminary observation.  

\begin{lem}   \label{lem:prelim obs}
	Let $\kappa>0$.  If $ \varepsilon' < \frac{\kappa}{4\modulus{E}^2}$ then  $|B_{0}|\geq (1-\kappa)|B|.$
\end{lem}

\begin{proof}
	Note that
\[B_{01}=\bigcap_{\substack{h_{1},h_{2}\in E\\h_1\ne h_2}}\{b\in B:\sigma_B(h_{1})b\ne \sigma_B(h_{2})b\},\]
\[B_{02}=\bigcap_{\substack{h_{1},h_{2}\in E\\h_1\ne h_2}}\{b\in B:\sigma_B(h_{1})\sigma_B(h_{2})b=\sigma_B(h_{1}h_{2})b\}.\]
So
\[\frac{\modulus{B\setminus B_{01}}}{\modulus{B}}\leq \sum_{\substack{h_{1},h_{2}\in E\\h_1\ne h_2}}\left(1-\frac{\modulus{ \{b\in B:\sigma_B(h_{1})b\ne \sigma_B(h_{2})b\}}}{\modulus{B}}\right).\]
We have for all $h_{1},h_{2}\in E$ with $h_{1}\ne h_{2}:$
\begin{eqnarray*}\frac{\modulus{\{b\in B:\sigma_{B}(h_{1})b\ne \sigma_{B}(h_{2})b\}}}{\modulus{B}}&=&d_{\Hamm}(\sigma_{B}(h_{1}),\sigma_{B}(h_{2}))\\
&=&d_{\Hamm}(\sigma_{B}(h_{2})^{-1}\sigma_{B}(h_{1}),\id).\\
\end{eqnarray*}
Since $d_{\Hamm}$ is a invariant under left multiplication, and $E_B\supseteq E\cup E\m$ we have that 
\[d_{\Hamm}(\sigma_{B}(h_{2})^{-1},\sigma_{B}(h_{2}^{-1}))=d_{\Hamm}(\id,\sigma_{B}(h_{2})\sigma_{B}(h_{2}^{-1}))<\varepsilon'.\]
Inserting this into the above two inequalities and using that $d_{\Hamm}$ is invariant under right multiplication we see that:
\begin{eqnarray*}\frac{\modulus{\{b\in B:\sigma_{B}(h_{1})b\ne \sigma_{B}(h_{2})b\}}}{\modulus{B}}&=&d_{\Hamm}(\sigma_{B}(h_{1}),\sigma_{B}(h_{2}))\\
&=&d_{\Hamm}(\sigma_{B}(h_{2})^{-1}\sigma_{B}(h_{1}),\id)\\
&> &d_{\Hamm}(\sigma_{B}(h_{2}^{-1})\sigma_{B}(h_{1}),\id)-\varepsilon'\\
&\geq& d_{\Hamm}(\sigma_{B}(h_{2}^{-1}h_{1}),\id)-2\varepsilon'\\
&> & 1-3\varepsilon',
\end{eqnarray*}
where in the last two lines we again use that  $E_B\supseteq E\cup E\m\cup E\m E$.  Thus
\[\frac{\modulus{B_{01}}}{\modulus{B}}\geq (1-3\modulus{E}^{2}\varepsilon').\]
Similarly, $(E_B,\varepsilon')$--multiplicativity of $\sigma_B$ gives
\[\frac{\modulus{B_{02}}}{\modulus{B}}\geq 1- \sum_{h_{1},h_{2}\in E}\Big(1-d_{\Hamm}\big(\sigma_{B}(h_{1}h_{2}),\sigma_{B}(h_{1})\sigma_{B}(h_{2})\big)\Big)\geq 1-\modulus{E}^{2}\varepsilon'.\]
This proves the Lemma.
\end{proof}

Fix $\kappa>0$ with  $\kappa<\frac{\varepsilon}{12}$ and choose $\varepsilon'>0$ as in Lemma \ref{lem:prelim obs},  so $\varepsilon' < \frac{\varepsilon}{48\modulus{E}^2}$. 
To complete the proof of Proposition \ref{prop:sofic approx}, we break it up into two steps: we first show that this choice of $\varepsilon'$ gives us that $\widehat{\sigma}$ is $(F,\varepsilon)$--multiplicative, and then show it  is $(F,\varepsilon)$--free.

\medskip

\textit{Step 1.} We show that  $\widehat{\sigma}$ is $(F,\varepsilon)$--multiplicative.

To prove Step 1, we apply Lemma \ref{l:almosthomom}, verifying below the four necessary conditions.
 We first check that the restriction to $\oplus_H G$ is $(E_1,\varepsilon/6)$--multiplicative. 
Let $g,g'\in E_{1},$ then
\begin{align*}
d_{\Hamm}&\big(\widehat{\sigma_{A}}(gg'),\widehat{\sigma_{A}}(g)\widehat{\sigma_{A}}(g')\big)\\
&\leq \frac{1}{|A|^{|B|}|B|}\modulus{\bigoplus_{B}A\oplus (B\setminus B_{0})}
+ \frac{\modulus{\big\{(a,b):b\in B_{0},\sigma_{A,b}(g)\sigma_{A,b}(g')\ne \sigma_{A,b}(gg')\big\}}}{|A|^{|B|}|B|}\\
&\leq \kappa+\frac{1}{|B|}\sum_{b\in B_{0}}d_{\Hamm}(\sigma_{A,b}(g)\sigma_{A,b}(g'),\sigma_{A,b}(gg')).
\end{align*}
Write $g=(g_x),g'=(g'_x)$.
 Recall, $\sigma_{A,b}(g) = \prod_{x\in E} \sigma_{A,b}^{(x)}(g_x)$.
In the following, we use that the different terms in the product commute, to be precise: for $x_1\neq x_2 \in E$ we have $[ \sigma_{A,b}^{(x_1)}(g_{x_1}) , \sigma_{A,b}^{(x_2)}(g'_{x_2}) ] = 1$. This gives
$$d_{\Hamm}(\sigma_{A,b}(g)\sigma_{A,b}(g'),\sigma_{A,b}(gg'))
	=	d_{\Hamm}\left(\prod\limits_{x\in E}\sigma_{A,b}^{(x)}(g_x)\sigma_{A,b}^{(x)}(g'_x) , \prod\limits_{x\in E} \sigma_{A,b}^{(x)}(g_xg'_x) \right).$$
Next, using the bi-invariance of the Hamming distance, the triangle inequality, and
	the $(E_A,\varepsilon')$--multiplicativity of $\sigma_A$ to see that
\begin{eqnarray*}d_{\Hamm}(\sigma_{A,b}(g)\sigma_{A,b}(g'),\sigma_{A,b}(gg'))
	&\leq &	\sum\limits_{x\in E} d_{\Hamm}\left(\sigma_{A,b}^{(x)}(g_x)\sigma_{A,b}^{(x)}(g'_x) ,  \sigma_{A,b}^{(x)}(g_xg'_x) \right)\\
	&=&\sum_{x\in E}d_{\Hamm}(\sigma_{A}(g_{x})\sigma_{A}(g_{x}'),\sigma_{A}(g_{x}g'_{x}))\\
	&<&\modulus{E}\varepsilon'.\end{eqnarray*}
Thus we get the required multiplicativity:
\begin{equation*}
d_{\Hamm}(\widehat{\sigma_{A}}(gg'),\widehat{\sigma_{A}}(g)\widehat{\sigma_{A}}(g'))<\kappa+\modulus{E}\varespilon' < \frac{\varepsilon}{6}.
\end{equation*}

The fact that the restriction to $H$ is $(E_2,\varepsilon/6)$--multiplicative is more straight-forward.
Indeed, for $h,h'\in E_{2}$ we have
\begin{equation*}
d_{\Hamm}(\widehat{\sigma_{B}}(hh'),\widehat{\sigma_{B}}(h)\widehat{\sigma_{B}}(h'))=d_{\Hamm}(\sigma_{B}(hh'),\sigma_{B}(h)\sigma_{B}(h'))<\varepsilon',
\end{equation*}
where we note that we can use the multiplicative property of $\sigma_B$ since $E_2 \subseteq E_B$.

The third condition of Lemma \ref{l:almosthomom} is automatically satisfied by $\widehat{\sigma}$, by construction.
We finish this step by verifying the bound on the Hamming distance between $\widehat{\sigma}(1,h)\widehat{\sigma}(g,1)$ and $\widehat{\sigma}(\alpha_h(g),1)\widehat{\sigma}(1,h)$ for $h\in E_{2},g\in E_{1}$. Indeed, for such $g,h$ we have
\begin{align*}
d_{\Hamm}&(\widehat{\sigma_{B}}(h)\widehat{\sigma_{A}}(g),\widehat{\sigma_{A}}(\alpha_{h}(g))\widehat{\sigma_{B}}(h))\\
&\leq \frac{1}{|A|^{|B|}|B|}\left|\left(\bigoplus_{B}A\right)\oplus B\setminus (B_{0}\cap \sigma_{B}(h)^{-1}(B_{0}))\right|\\
&\phantom{aaaaaa} + \frac{1}{|A|^{|B|}|B|}\modulus{\big\{(a,b):b\in B_{0}\cap \sigma_{B}(h)^{-1}B_{0},\sigma_{A,b}(g)a \ne \sigma_{A,\sigma_{B}(h)b}(\alpha_{h}(g))a\big\}}\\
&\leq 2\kappa + \frac{1}{|A|^{|B|}|B|}\modulus{\big\{(a,b):b\in B_{0}\cap \sigma_{B}(h)^{-1}B_{0},\sigma_{A,b}(g)a \ne \sigma_{A,\sigma_{B}(h)b}(\alpha_{h}(g))a\big\}}.
\end{align*}
Since $\supp(\alpha_{h}(g))=h\supp(g)$, and $E$ contains both $\Supp(g)$ and $h\supp(g)$, it follows that for every $b\in B_{0}\cap \sigma_{B}(h)^{-1}(B_{0})$ we have
\[\sigma_{A,\sigma_{B}(h)b}(\alpha_{h}(g))=\prod_{x\in h\supp(g)}\sigma_{A,\sigma_{B}(h)b}^{(x)}(g_{h^{-1}x})=\prod_{x\in \supp(g)}\sigma_{A,\sigma_{B}(h)b}^{(hx)}(g_{x}).\]
Note that we have used that $\sigma_A(1) = \id$ to restrict the number of terms in the product.
We use that for $h\in E$ (and hence for $h\in E_2$) we have that $\sigma_{A,\sigma_{B}(h)b}^{(hx)}(g)=\sigma_{A,b}^{(x)}(g)$.
Inserting this into the above equation we see that
\[ \sigma_{A,\sigma_{B}(h)b}(\alpha_{h}(g))= \prod_{x\in \supp(g)}\sigma_{A,b}^{(x)}(g)= \sigma_{A,b}(g).\]
Returning to the above inequality, we have shown that the remaining sets involved are empty, implying that 
\begin{equation*}
d_{\Hamm}(\widehat{\sigma_{B}}(h)\widehat{\sigma_{A}}(g),\widehat{\sigma_{A}}(\alpha_{h}(g))\widehat{\sigma_{B}}(h))<2\kappa<\frac{\varepsilon}{6}.
\end{equation*}
The hypotheses of Lemma \ref{l:almosthomom} have been checked, completing Step 1.

\medskip

\textit{Step 2.} We show that $\widehat{\sigma}$ is $(F,\varepsilon)$--free.

Suppose $(g,h)\in F.$ If $h\ne 1,$ then
\[\big\{(a,b):\widehat{\sigma}(g,h)(a,b)\ne (a,b)\big\}\supseteq \left(\bigoplus_{B}A\right)\oplus \big\{b\in B:\sigma_{B}(h)b\ne b\big\}.\]
Thus
\[d_{\Hamm}\big(\widehat{\sigma}(g,h),\id\big)\geq d_{\Hamm}\big(\sigma_{B}(h),\id\big)\geq 1-\varepsilon'> 1-\varepsilon.\]
We may therefore assume that $h=1.$
When $b\in B_0$, the permutation $\widehat{\sigma}(g,1)$ will act by $(a,b)\mapsto (\sigma_{A,b}(g)a,b)$, for all $a \in \oplus_B A$.
We then see that
\[\big\{(a,b):\widehat{\sigma}(g,1)(a,b)=(a,b)\big\} \subseteq \left( \left( \bigoplus_{B}A\right) \oplus \big(B\setminus B_{0} \big) \right) \cup \big\{(a,b):b\in B_{0},\sigma_{A,b}(g)a=a\big\}.\]
 Assume $g=(g_x)\neq 1$ and consider the proportion of elements fixed by $\widehat{\sigma}(g,1)$. By the above we have
\[\frac{1}{|A|^{|B|}|B|}\modulus{ \big\{ (a,b) : \widehat{\sigma}(g,1)(a,b) = (a,b) \big\}}
\leq \kappa + \frac{1}{|B|}\sum_{b\in B_{0}}\frac{\modulus{\big\{ a\in \bigoplus_{B}A:\sigma_{A,b}(g)a=a\big\}}}{|A|^{|B|}}.\]
Since $g\ne 1,$ we can find $x_{0}\in \supp(g).$ For every $b\in B_{0},$ we have that $\sigma_{B}(h_{1})b\ne \sigma_{B}(h_{2})b$ for every $h_{1},h_{2}\in \supp(g)$ with $h_{1}\ne h_{2}.$
It follows that the $(\sigma_{B}(x_{0})b)$--coordinate of $\sigma_{A,b}(g)a$ is $\sigma_{A}(g_{x_0})a_\beta$, where $a_\beta$ is the $(\sigma_{B}(x_{0})b)$--coordinate of $a$.
Thus
\[\left\{a\in \bigoplus_{B}A:\sigma_{A,b}(g)a=a\right\}\subseteq \left(\bigoplus_{B\setminus \{\sigma_{B}(x_{0})b\}}A\right) \oplus \big\{a\in A:\sigma_{A}(g_{x_0})a=a \big\}.\]
So
\[\frac{\modulus{ \big\{ a\in \bigoplus_{B}A:\sigma_{A,b}(g)a=a\big\}}}{|A|^{|B|}} \leq 1-d_{\Hamm}(\sigma_{A}(g_{x_0}),1).\]
Since $g_{x_0} \in E_{A}$ we can use the freeness of $\sigma_A$, yielding 
\[\frac{1}{|A|^{|B|}|B|}|\{(a,b):\widehat{\sigma}(g,1)(a,b)=(a,b)\}|\leq\kappa+\frac{|B_{0}|}{|B|}(1-d_{\Hamm}(\sigma_{A}(g),1))\leq \kappa+\varepsilon'.\]
Equivalently,
\[d_{\Hamm}(\widehat{\sigma}(g,1),\id)\geq 1-\kappa-\varepsilon' > 1-\varepsilon.\]
This verifies that $\widehat{\sigma}$ is $(F,\varepsilon)$--free, and thus completes the proof of Proposition \ref{prop:sofic approx}, and hence of Theorem \ref{thmspecial:sofic wreath}.

\end{document}